\documentclass[12pt,reqno]{amsart}

\usepackage[headings]{fullpage} 
\usepackage{times}
\usepackage[T1]{fontenc}
\usepackage{mathrsfs}

\numberwithin{equation}{section}

\theoremstyle{plain}
\newtheorem{thm}[equation]{Theorem}
\newtheorem*{thm*}{Theorem}
\newtheorem{prop}[equation]{Proposition}

\newtheorem{lemma}[equation]{Lemma}
\newtheorem{conj}[equation]{Conjecture}

\theoremstyle{definition}

\theoremstyle{remark}
\newtheorem{rem}[equation]{Remark}

\newcommand{\Fq}{{\mathbb{F}_q}}
\newcommand{\U}{{\mathbb{U}}}
\DeclareMathOperator{\rk}{Rank}
\def\nmid{\mathrel{\mathchoice{\not|}{\not|}{\kern-.2em\not\kern.2em|}{\kern-.2em\not\kern.2em|}}}
\newcommand{\1}{{\bf 1}}

\renewcommand{\mod}[1]{{\ifmmode\text{\rm\ (mod~$#1$)}\else\discretionary{}{}{\hbox{ }}\rm(mod~$#1$)\fi}}

\newcommand{\li}{{\rm li\kern1pt}}

\newcommand{\A}{{\mathcal A}}
\newcommand{\B}{{\mathcal B}}

\newcommand{\e}{{\rm e}}

\newcommand{\C}{{\mathbb C}}

\newcommand{\R}{{\mathcal R}}

\newcommand{\Z}{{\mathbb Z}}

\begin{document}

\title{On balanced subgroups of the multiplicative group}

\author{Carl Pomerance}
\author{Douglas Ulmer}

\address{Department of Mathematics, Dartmouth College, Hanover, NH~~03755}
\email{carl.pomerance@dartmouth.edu}

\address{School of Mathematics, Georgia Institute of Technology, Atlanta, GA~~30332}
\email{douglas.ulmer@math.gatech.edu}

\thanks{The first author was partially supported by NSF grant DMS-1001180.}

\dedicatory{In memory of Alf van der Poorten}

\begin{abstract}
  A subgroup $H$ of $G=(\Z/d\Z)^\times$ is called \emph{balanced} if
  every coset of $H$ is evenly distributed between the lower and upper
  halves of $G$, i.e., has equal numbers of elements with
  representatives in $(0,d/2)$ and $(d/2,d)$.  This notion has
  applications to ranks of elliptic curves.  We give a simple
  criterion in terms of characters for a subgroup $H$ to be balanced,
  and for a fixed integer $p$, we study the distribution of
  integers $d$ such that the cyclic subgroup of $(\Z/d\Z)^\times$
  generated by $p$ is balanced.
\end{abstract}

\maketitle

\section{Introduction}

Let $d>2$ be an integer and consider $\U_d=(\Z/d\Z)^\times$, the group of
units modulo $d$.  Let $A_d$ be the first half of $\U_d$; that is, $A_d$
consists of residues with a representative in $(0,d/2)$.  Let
$B_d=\U_d\setminus A_d$ be the second half of $\U_d$.  We say a subgroup
$H$ of $\U_d$ is {\it balanced} if for each $u\in \U_d$ we have $|uH\cap
A_d|=|uH\cap B_d|$; that is, each coset of $H$ has equally many
members in the first half of $\U_d$ as in the second half.

Our interest in this notion stems from the following result.

\begin{thm}[\cite{CHU}]\label{thm:CHU}
  Let $p$ be an odd prime number, let $\Fq$ be the finite field of
  cardinality $q=p^f$, and let $\Fq(u)$ be the rational function
  field over $\Fq$.  Let $d$ be a positive integer not divisible by
  $p$, and for $e$ a divisor of $d$ write $\langle p\rangle_e$ for the cyclic subgroup of
  $(\Z/e\Z)^\times$ generated by $p$.  Let $E_d$ be the elliptic curve over $\Fq(u)$ defined by
$$y^2=x(x+1)(x+u^d).$$
Then we have
$$\rk E_d(\Fq(u))=\sum_{\substack{e|d\\ e>2}}
\begin{cases}
\frac{\varphi(e)}{l_q(e)}&\text{if $\langle p\rangle_e$ is balanced,}\\
0&\text{if not.}
\end{cases}$$
Here $\varphi$ is Euler's function and $l_q(e)$ is the order of $q$ in
$(\Z/e\Z)^\times$.
\end{thm}

A few simple observations are in order.  It is easy to see that
$\langle-1\rangle$ is a balanced subgroup of $\U_d$.  It is also
easy to see that if $4\mid d$, then $\langle\frac12d+1\rangle$ is a
balanced subgroup of $\U_d$.  In addition, if $H$ is a balanced subgroup
of $\U_d$ and $K$ is a subgroup of $\U_d$ containing $H$, then $K$ is
balanced as well.  Indeed, $K$ is a union of $[K:H]$ cosets of $H$, so
for each $u\in \U_d$, $uK$ is a union of $[K:H]$ cosets of $H$, each
equally distributed between the first half of $\U_d$ and the second
half.  Thus, $uK$ is also equally distributed between the first half
and the second half.

It follows that if some power of $p$ is congruent to $-1$ modulo $d$
and if $q\equiv1\pmod d$, then the theorem implies that $\rk E(\Fq(u))=d-2$
if $d$ is even and $d-1$ if $d$ is odd.  The rank of $E$ when some
power of $p$ is $-1$ modulo $d$ was first discussed in \cite{U-L1},
and with hindsight it could have been expected to be large from
considerations of ``supersingularity.''  The results of \cite{CHU}
show, perhaps surprisingly, that there are many other classes of $d$
for which high ranks occur.  Our aim here is to make this observation
more quantitative.

More precisely, the aim of this paper is to investigate various
questions about balanced pairs $(p,d)$, i.e., pairs such that $\langle
p\rangle$ is a balanced subgroup of $(\Z/d\Z)^\times$.  In particular,
we give a simple criterion in terms of characters for $(p,d)$ to be
balanced (Theorem~\ref{thm:char-criterion}), and we use it to
determine all balanced subgroups of order 2
(Theorem~\ref{thm:order2}).  We also investigate the distribution for
a fixed $p$ of the set of $d$'s such that $(p,d)$ is balanced
(Theorems~\ref{thm:dist0},~\ref{thm:dist1}).  
We find that when $p$ is odd, the divisors
of $p^n+1$ are not the largest contributor to this set.  Finally, we
investigate the average rank and typical rank given by
Theorem~\ref{thm:CHU} for fixed $q$ and varying $d$.

\section{Balanced subgroups and characters}
In this section we write $G$ (rather than $\U_d$) for
$(\Z/d\Z)^\times$.  We also write $A$ for $A_d$ as above and similarly
for $B$, so that $G$ is the disjoint union $A\cup B$.

We write $\1_A$ for the characteristic function of $A\subset G$ and
similarly for $\1_B$.  Let $f:G\to\C$ be the sum over $H$ of translates
of $\1_A-\1_B$:
\begin{align*}
f(g)&=\sum_{h\in H} \left(\1_A(gh)-\1_B(gh)\right)\\
&=\#(gH\cap A)-\#(gH\cap B).
\end{align*}
By definition, $H$ is balanced if and only if $f$ is identically zero.

We write $\hat G$ for the set of complex characters of $G$, and 
we expand $f$ in terms of these characters:
$$f=\sum_{\chi\in \hat G}\hat f(\chi)\chi$$
where
$$\hat f(\chi)=\frac1{\varphi(d)}\sum_{g\in G}f(g)\chi^{-1}(g).$$

It is easy to see that $\hat f(\chi_{triv})=0$.  Noting that
$\1_A-\1_B=2\1_A-\1_G$, for $\chi$ non-trivial we find that
$$\hat f(\chi^{-1})=\frac2{\varphi(d)}\left(\sum_{h \in H} \chi(h)\right)
\left(\sum_{a\in A} \chi(a)\right).$$

Note that $\sum_{h \in H} \chi(h)$ is zero if and only if the
restriction of $\chi$ to $H$ is non-trivial.  We introduce the
notation
$$c_\chi=\sum_{a\in A} \chi(a).$$
We view $\chi$ as a Dirichlet character, so that we also have
$$c_\chi=\sum_{0<a<d/2}\chi(a).$$

As usual, we say $\chi$ is even if $\chi(-1)=1$ and $\chi$ is odd if
$\chi(-1)=-1$.  Note that if $\chi$ is even and non-trivial, then
$$c_\chi=\frac12\sum_{g\in G}\chi(g)=0.$$

This discussion yields the following characterization of balanced subgroups:

\begin{thm}\label{thm:char-criterion}
With notation as above, we have that $H$ is balanced if and only if
$c_\chi=0$ for every odd character $\chi$ of $G$ whose restriction to $H$
is trivial.
\end{thm}
\qed

As an example, note that if $H=\langle-1\rangle$, then there are no odd
characters trivial on $H$ and so the theorem implies that $H$ is balanced.

We now give a non-vanishing criterion for $c_\chi$.

\begin{lemma}
If $\chi$ is a primitive, odd character of $G$, then $c_\chi\neq0$.
\end{lemma}

\begin{proof}
  Under the hypotheses on $\chi$, the classical evaluation of
  $L(1,\chi)$ leads to the formula
$$L(1,\chi^{-1})=\frac{\pi
  i\tau(\chi^{-1})}{d(\chi^{-1}(2)-2)}c_\chi$$ 
where $\tau(\chi^{-1})$ is a Gauss sum. 
(See, e.g., \cite[pp.~200--201]{Marcus} or \cite[Theorem~9.21]{MV}, though
there is a small typo in the second reference.)
By the theorem of Dirichlet, $L(1,\chi^{-1})\neq0$ and so $c_\chi\neq0$.
\end{proof}

In light of the lemma, we should consider imprimitive characters.

\begin{lemma}\label{lemma:imprimitive}
Suppose that $\ell$ is a prime number dividing $d$ and set
$d'=d/\ell$.  Suppose also that $\chi$ is a non-trivial character modulo $d$
induced by a character $\chi'$ modulo $d'$.  
If $\ell=2$, then
$c_\chi=-\chi'(2)c_{\chi'}$.
If $\ell$ is odd, then
$c_\chi=(1-\chi'(\ell))c_{\chi'}$.
Here we employ the usual convention that $\chi'(\ell)=0$ if $\ell\mid d'$.
\end{lemma}

\begin{proof}
First suppose $\ell=2$.  We have
$$c_\chi=\sum_{\substack{a<d/2\\\gcd(a,d)=1}}\chi(a)=
\sum_{\substack{a<d'\\\gcd(a,2d')=1}}\chi'(a).$$
If $2\mid  d'$, this is a complete character sum and so vanishes.
If $2\nmid d'$, then
\begin{align*}
\sum_{\substack{a<d'\\\gcd(a,2d')=1}}\chi'(a)
&=\sum_{\substack{a<d'\\\gcd(a,d')=1}}\chi'(a)
-\sum_{\substack{a<d'/2\\\gcd(a,d')=1}}\chi'(2a)\\
&=-\sum_{\substack{a<d'/2\\\gcd(a,d')=1}}\chi'(2a)\\
&=-\chi'(2)c_{\chi'}
\end{align*}
as desired.

Now assume that $\ell$ is odd.  We have
$$c_\chi=\sum_{\substack{a<d/2\\\gcd(a,d)=1}}\chi(a)=
\sum_{\substack{a<\ell d'/2\\\gcd(a,\ell d')=1}}\chi'(a).$$
If $\ell\mid d'$, then
$$\sum_{\substack{a<\ell d'/2\\\gcd(a,\ell d')=1}}\chi'(a)
=\sum_{\substack{a<d'/2\\\gcd(a,d')=1}}\chi'(a)
=c_{\chi'}.$$
If $\ell\nmid d'$, then
\begin{align*}
\sum_{\substack{a<\ell d'/2\\\gcd(a,\ell d')=1}}\chi'(a)
&=\sum_{\substack{a<\ell d'/2\\\gcd(a,d')=1}}\chi'(a)
-\sum_{\substack{a<d'/2\\\gcd(a,d')=1}}\chi'(\ell a)\\
&=\sum_{\substack{a<d'/2\\\gcd(a,d')=1}}\chi'(a)
-\chi'(\ell)\sum_{\substack{a<d'/2\\\gcd(a,d')=1}}\chi'(a)\\
&=(1-\chi'(\ell))c_{\chi'}
\end{align*}
as desired.
\end{proof}

Applying the lemma repeatedly, we arrive at the following non-vanishing
criterion.

\begin{prop}\label{prop:non-vanishing}
  Suppose that $\chi$ is an odd character modulo $d$ induced by a
  primitive character $\chi'$ modulo $d'$.  Then $c_\chi\neq0$ if and
  only if the following two conditions both hold: \textup{(}i\textup{)} $4\nmid d$ or
  $d/d'$ is odd; and \textup{(}ii\textup{)} for every odd prime $\ell$ which divides $d$
  and does not divide $d'$, we have $\chi'(\ell)\neq1$.
\end{prop}

As an example, suppose that $4\mid d$ and $H=\langle\frac12d+1\rangle$.
Note that 
$$(\Z/d\Z)^\times/\langle\textstyle{\frac12}d+1\rangle
\cong(\Z/\textstyle{\frac12}d\Z)^\times.$$ 
Thus, if $\chi$ is an odd character modulo $d$ and $\chi(\frac12d+1)=1$,
then the conductor $d'$ of $\chi$ divides $d/2$.  This shows that
$d/d'$ is even and so condition (i) of the proposition fails and
$c_\chi=0$.  Therefore $H$ is balanced.

\section{Balanced subgroups of small order}
In this section, we discuss balanced subgroups of small order.  We
have already seen that a subgroup of $G$ which contains $-1$ or
$\frac12d+1$ is balanced.  We will show that in a certain sense small
balanced subgroups are mainly controlled by these balanced subgroups
of order 2.

\begin{thm}\label{thm:small-order}
  For every positive integer $n$ there is an integer $d(n)$ such that
  if $d>d(n)$ and $H$ is a balanced subgroup of $G=(\Z/d\Z)^\times$ of
  order $n$, then either $-1\in H$ or $4\mid d$ and $\frac12d+1\in H$.
\end{thm}

We can make this much more explicit for subgroups of order 2:

\begin{thm}\label{thm:order2}
  A subgroup $H=\langle h\rangle$ of order 2 is balanced if and only if $d$
  and $h$ satisfy one of the following conditions:
\begin{enumerate}
\item $h\equiv-1\pmod d$
\item $d\equiv0\pmod 4$ and $h\equiv\frac12d+1\pmod d$
\item $d=24$ and $h\equiv 17\pmod d$ or $h\equiv19\pmod d$.
\item $d=60$ and $h\equiv 41\pmod d$ or $h\equiv49\pmod d$.
\end{enumerate}
\end{thm}

\begin{proof}[Proof of Theorem~\ref{thm:small-order}]
  Using Proposition~\ref{prop:non-vanishing}, we will show that if $d$
  is sufficiently large with respect to $n$, then for any subgroup
  $H\subset(\Z/d\Z)^\times$ of order $n$ which does not contain $-1$
  or $\frac12d+1$, there is a character $\chi$ which is odd, trivial on
  $H$, and with $c_\chi\neq0$.  By Theorem~\ref{thm:char-criterion},
  this implies that $H$ is not balanced.

  Note that a balanced subgroup obviously has even order, so there is
  no loss in assuming that $n$ is even.  We make this assumption for
  the rest of the proof.

  Let $H^+$ be the subgroup of $G$ generated by $H$, $-1$ and, if
  $4\mid d$, by $\frac12d+1$.  Fix a character $\chi_0$ of $G$ which is
  trivial on $H$, odd, and $-1$ on $\frac12d+1$ if $4\mid d$.  The set of all
  characters satisfying these restrictions is a homogeneous space for
  $\widehat{G/H^+}\subset\widehat G$.  We will argue that multiplying
  $\chi_0$ by a suitable $\psi\in \widehat{G/H^+}$ yields a
  $\chi=\chi_0\psi$ for which Proposition~\ref{prop:non-vanishing}
  implies that $c_\chi\neq0$.

  Note first that any character $\chi$ which is odd and, if $4\mid d$,
  has $\chi(\frac12d+1)=-1$ automatically satisfies condition (i) in
  Proposition~\ref{prop:non-vanishing}.  Indeed, if $4\mid d$, then
  the condition $\chi(\frac12d+1)=-1$ implies that $\chi$ is
  2-primitive, i.e., the conductor $d'$ of $\chi$ has $d/d'$ odd.  The
  rest of the argument relates to condition (ii) in
  Proposition~\ref{prop:non-vanishing}.

  Write $d=\prod_{\ell}\ell^{e_\ell}$ and write $G_\ell$ for
  $(\Z/\ell^{e_\ell}\Z)^\times$ so that $G\cong\prod_\ell G_\ell$.
  Let $\chi=\prod_\ell\chi_\ell$.  Note that $\ell$ divides the
  conductor of $\chi$ if and only if $\chi_\ell$ is non-trivial.

  We will sloppily write $G_\ell/H^+$ for $G_\ell$ modulo the image of
  $H^+$ in $G_\ell$.  For odd $\ell$, $G_\ell$ is cyclic and therefore
  so is $G_\ell/H^+$; for $\ell=2$, since $-1\in H^+$, $G_2/H^+$ is
  also cyclic.  

  Note also that $H^+$ is the product of $H$ and a group of exponent
  2, namely the subgroup of $G$ generated by $-1$ or by $-1$ and
  $\frac12d+1$.  Also, we have assumed that $n=|H|$ is even.  If $\ell$
  is odd, then $G_\ell$ is cyclic of even order, so has a unique
  element of order 2.  It follows that the order of the image of $H^+$
  in $G_\ell$ divides $n$.

  We define three sets of odd primes:
$$S_1=\left\{\text{odd }\ell: \ell\mid d,~ G_\ell/H^+=\{1\}\right\},$$
$$S_2=\left\{\text{odd }\ell:\ell\mid d,~\varphi(\ell^{e_\ell})\mid n\right\},$$
and
$$S_3=\left\{\text{odd }\ell:\varphi(\ell)\mid n\right\}.
$$
Note that $S_1\subset S_2\subset S_3$ and $S_3$ depends only on $n$,
not on $d$.

If $\ell$ is odd, $\ell\mid d$, and $\ell\not\in S_1$, then
$G_\ell/H^+$ is non-trivial.  Thus choosing a suitable $\psi$, we may
arrange that the conductor of $\chi_1=\chi_0\psi$ is divisible by
every prime dividing $d$ which is not in $S_1$.

For the odd primes $\ell$ which divide $d$ and do not divide the
conductor of $\chi_1$ (a subset of $S_1$, thus also a subset of
$S_3$), we must arrange that $\chi'(\ell)\neq1$ (where $\chi'$ is the
primitive character inducing $\chi$).

Recall that $G_\ell/H^+$ is cyclic.  We now remark that if $C$ is a
cyclic group and $a\in C$, and $z\in\C$, then the set of characters
$\psi:C\to\C$ such that $\chi(a)\neq z$ has cardinality at least
$|C|(1-1/|\langle a\rangle|)$ (where $|\langle a\rangle|$ is the order
of $a$).  If we have several elements $a_1,\dots,a_n$ and several
values $z_1,\dots,z_n$ to avoid, then the number of characters $\psi$
such that $\psi(a_i)\neq z_i$ is at least 
$$|C|\left(1-\frac{1}{|\langle a_1\rangle|}-\cdots-\frac{1}{|\langle a_n\rangle|}\right).$$
Thus we can find such a character
provided that each $a_i$ has
order $>n$.

Now we use that $d$ is large to conclude that a large prime power
$\ell^e$ divides $d$.  (Note that $\ell$ might be 2 here.)  Then
$G_\ell/H^+$ is a cyclic group in which the order of each prime in $S_1$
is large. (The primes in $S_1$ are also in $S_3$, so belong to a set fixed
independently of $d$.)  We want a character $\psi$ of $G_\ell/H^+$
which satisfies $\psi(r)\neq\chi_1^{-1}(r)$ for all $r\in S_1$.  We also
want $\psi\chi_1$ to have non-trivial $\ell$ component which, phrased
in the language above, means that we want $\psi(a)\neq1$ for some
fixed generator of $G_\ell/H^+$.  Since the size of $S_1$ is fixed
depending only on $n$, the discussion of the previous paragraph shows
that these conditions can be met if $\ell^e$ is large enough.

Setting $\chi=\psi\chi_1$ with $\psi$ as in the previous paragraph
yields a character $\chi$ such that $c_\chi\neq0$, and this completes
the proof.
\end{proof}

\begin{proof}[Proof of Theorem~\ref{thm:order2}]
We retain the concepts and notation of the proof of
Theorem~\ref{thm:small-order}.  We also say that a subgroup
of order 2 is ``exceptional'' if it does not contain $-1$ or $\frac12d+1$.

Since $n=2$, the set $S_3=\{3\}$ and the set $S_2$ is either empty (if
$3\nmid d$ or $9\mid d$) or $S_2=\{3\}$ (if $3$ exactly divides $d$).  If
$S_2$ is empty and $H$ is an exceptional subgroup of order 2, then the
first part of the proof of Theorem~\ref{thm:small-order} provides a
primitive odd character trivial on $H$, and so $H$ is not balanced.

Suppose we are in the case where $3$ exactly divides $d$.  Following
the first part of the proof of Theorem~\ref{thm:small-order}, we have
a character $\chi_1$ of $G$ with conductor divisible by $d'=d/3$ which
is odd, trivial on $H$, and, if $4\mid d$, satisfies $\chi_1(\frac12d+1)=-1$.
If the conductor of $\chi_1$ is $d$ or if the primitive character
$\chi'$ inducing $\chi_1$ has $\chi'(3)\neq1$, then setting
$\chi=\chi_1$ we have $c_\chi\neq0$ and we see that $H$ is not
balanced.

If not, we will modify $\chi_1$.  Note that if $\ell=2$ and $16\mid
d$, or $\ell=5$ and $25\mid d$, or $\ell$ is a prime $\ge7$ and
$\ell\mid d$, then the order of $3$ in $G_\ell/H^+$ is at least 3.
Thus in these cases, there is a character $\psi$ of $G_\ell/H^+$ so
that the $\ell$ part of $\chi=\chi_1\psi$ is non-trivial and so that
the primitive character $\chi'$ inducing $\chi$ satisfies
$\chi'(3)\neq1$.  Then $c_\chi\neq0$ and $H$ is not balanced.

This leaves a small number of values of $d$ to check for exceptional
balanced subgroups of order $2$.  Namely, we just need to check
divisors of $8\cdot3\cdot5=120$ which are divisible by $3$.  A quick
computation which we leave to the reader finishes the proof.
\end{proof}

\section{Distribution of numbers $d$ with $\langle p\rangle_d$ balanced}
Fix an integer $p$ with $|p|>1$.  
In our application to elliptic curves,
$p$ is an odd prime number, but it seems interesting to state
our results on balanced subgroups in a more general context.
Let $\B_p$ denote the set of integers $d>2$ coprime to $p$
for which $\langle p\rangle_d$ is a balanced subgroup of $\U_d$.
Further, let
\begin{align*}
\B_{p,0}&=\{d>2:(d,p)=1,~4\mid d,~\textstyle{\frac12}d+1\in\langle p\rangle_d\},\\
\B_{p,*}&=\B_p\setminus\B_{p,0},\\
\B_{p,1}&=\{d>2:(d,p)=1,~-1\in\langle p\rangle_d\}.
\end{align*}
Note that if $p$ is even then $\B_{p,0}$ is empty.
For any set $\A$ of positive integers and $x$ a real
number at least~1, we let $\A(x)=|\A\cap[1,x]|$.


We state the principal results of this section, which show that when
$p$ is odd, most members of $\B_p$ lie in $\B_{p,0}$.
\begin{thm}
\label{thm:dist0}
For each odd integer $p$ with $|p|>1$, 
there are positive numbers $b_p,b_p'$ with
$$
b_p\frac{x}{\log\log x}\le\B_{p,0}(x)\le b_p'\frac{x}{\log\log x}
$$ 
for all sufficiently large numbers $x$ depending on the choice of $p$.
\end{thm}
\begin{thm}
\label{thm:dist1}
For each integer $p$ with $|p|>1$, there is a number $\epsilon_p>0$
such that for all $x\ge3$,
$$
\B_{p,*}(x)=O_p\left(\frac{x}{(\log x)^{\epsilon_p}}\right).
$$
\end{thm}

We remark that $\B_{p,1}$ has been studied by Moree.
In particular we have the following result.
\begin{thm}[{\cite [Thm.~5]{Moree}}]
\label{thm-moree}
For each integer $p$ with $|p|>1$ there are positive numbers $c_p,\delta_p$
such that
$$
\B_{p,1}(x)\sim c_p\frac{x}{(\log x)^{\delta_p}},~x\to\infty.
$$
\end{thm}
Note that for $p$ prime we have $\delta_p=\frac23$.
We believe that $\B_{p,0}$ and $\B_{p,1}$ comprise most of $\B_p$
and in fact we pose the following conjecture.
\begin{conj}
\label{conj-mainconj}
For each integer $p$ with $|p|>1$ we have
$$
\B_p(x)=\B_{p,0}(x)+(1+o(1))\B_{p,1}(x),~x\to\infty,
$$
that is, $\B_{p,*}(x)\sim\B_{p,1}(x)$ as $x\to\infty$.
\end{conj}
It is easy to see that $\B_{p,1}\cap\B_{p,0}$ has at most one element.
Indeed, the cyclic group $\langle p\rangle_d$ has at most one element
of order exactly 2, so if $d\in \B_{p,1}\cap\B_{p,0} $, then for some
$f$ we have $p^f\equiv-1\equiv \frac12d+1\mod d$ and this can happen only
when $d=4$.  This at least shows that
$\B_p(x)\ge\B_{p,0}(x)+(1+o(1))\B_{p,1}(x)$ as $x\to\infty$.

\bigskip
We now begin a discussion leading to the proofs of
Theorems~\ref{thm:dist0} and \ref{thm:dist1}.
The following useful result comes from \cite[Theorem 2.2]{HR}.
\begin{prop}
\label{prop-sieve}
There is an absolute positive constant $c$ such that for all
numbers $x\ge3$ and any set $\R$ of primes in $[1,x]$, the number
of integers in $[1,x]$ not divisible by any member of $\R$ is at most
$$
cx\prod_{r\in \R}\left(1-\frac1r\right)\le 
cx\exp\left(-\sum_{r\in \R}\frac1r\right).
$$
\end{prop}
Note that the inequality in the display follows immediately from the 
inequality $1-\theta<{\rm e}^{-\theta}$ for every $\theta\in(0,1)$.

For a positive integer $m$ coprime to $p$, recall that $l_p(m)$ denotes the
order of $\langle p\rangle_m$.  If $r$ is a prime,
we let $v_r(m)$ denote that integer $v$ with $r^v\mid m$ and
$r^{v+1}\nmid m$.

We would like to give a criterion for membership in $\B_{p,0}$,
but before this, we establish an elementary lemma.
\begin{lemma}\label{lem-2power}
Let $p$ be an odd integer with $|p|>1$ and let $k,i$ be positive integers.
Then
$$
v_2\left(\frac{p^{2^ik}-1}{p^{2k}-1}\right)=i-1.
$$
\end{lemma}
\begin{proof}
The result is clear if $i=1$.  If $i>1$, we see that
$$
\frac{p^{2^ik}-1}{p^{2k}-1}=(p^{2k}+1)(p^{4k}+1)\dots(p^{2^{i-1}k}+1),
$$
which is a product of $i-1$ factors that are each 2~(mod~4).
\end{proof}

The following result gives a criterion for membership in $\B_{p,0}$.
\begin{prop}
\label{prop-crit}
Let $p$ be odd with $|p|>1$ and let $m\ge1$ be an odd integer coprime
to $p$.  If $l_p(m)$ is odd, then $2^jm\in\B_{p,0}$ if and only if
$j=1+v_2(p-1)$ or $j>v_2(p^2-1)$.  If $l_p(m)$ is even, then
$2^jm\in\B_{p,0}$ if and only if $j>v_2(p^{l_p(m)}-1)$.
\end{prop}
\begin{proof}
  We first prove the ``only if" part.  Assume that $d=2^jm\in\B_{p,0}$
  and let $f$ be an integer with $p^f\equiv \frac12d+1\pmod d$.  Then
  $l_p(m)\mid f$ so that $j-1=v_2(p^f-1)\ge v_2(p^{l_p(m)}-1)$.  This
  establishes the ``only if" part if $l_p(m)$ is even, and it also
  shows that $j\ge1+v_2(p-1)$ always, so in particular if $l_p(m)$ is
  odd.  Suppose $l_p(m)$ is odd and $1+v_2(p-1)<j\le v_2(p^2-1)$.
  Then $j-1>v_2(p-1)$, so that $l_p(2^{j-1}m)$ is even.  Using
  $l_p(m)$ odd, this implies that $2l_p(m)\mid f$, so that $2^j\mid
  (p^2-1)\mid (p^f-1)$, contradicting $p^f\equiv \frac12d+1\pmod d$.

  Towards showing the ``if" part, let $v=v_2(p^{l_p(m)}-1)$.  We have
  $p^{l_p(m)}-1\equiv 2^vm\pmod{2^{v+1}m}$, so that
  $2^{v+1}m\in\B_{p,0}$.  If $j>v+1$ and $l_p(m)$ is even, then with
  $f=2^{j-v-1}l_p(m)$, Lemma~\ref{lem-2power} implies that
  $p^f-1\equiv2^{j-1}m\pmod{2^jm}$, so that $2^jm\in\B_{p,0}$.  If
  $l_p(m)$ is odd, then $v=v_2(p-1)$, so that $2^{v+1}\in\B_{p,0}$.
  Finally assume that $j>v_2(p^2-1)$ and $l_p(m)$ is odd.  Then
  Lemma~\ref{lem-2power} implies that
  $p^{2^{j-v_2(p^2-1)}l_p(m)}-1\equiv2^{j-1}m\pmod{2^jm}$, so that
  $2^jm\in\B_{p,0}$.  This concludes the proof.
\end{proof}

\begin{proof}[Proof of Theorem~\ref{thm:dist0}]
For $m\ge1$ coprime to $2p$, let
$$
f_p(m):=v_2\left(p^{l_p(m)}-1\right),\quad
f_p'(m):=\max\left\{f_p(m),\,v_2\left(p^2-1\right)\right\}.
$$
Proposition~\ref{prop-crit} implies that if $2^jm\in \B_{p,0}$ with $m$ odd,
then $j>f_p(m)$.  Further, if $(m,2p)=1$ 
then $2^jm\in\B_{p,0}$
for all $j>f_p'(m)$.  

Using this last property, we have $\B_{p,0}(x)$ at least as big as the
number of choices for $m$ coprime to $2p$ with 
$1<m\le x/2^{f_p'(m)+1}$.  
Thus, the lower bound in the theorem will follow if we show that
there are at least $b_px/\log\log x$ integers $m$ coprime to $2p$
with $m\le x/2^{f_p'(m)+1}$.

Let $\lambda(m)$ denote Carmichael's function
at $m$, which is the order of the largest cyclic subgroup of $\U_m$.
Then $l_p(m)\mid\lambda(m)$.  Also, for $m>2$, $\lambda(m)$ is even,
so that 
$$
f_p'(m)\le g_p(m):=v_2\left(p^{\lambda(m)}-1\right).
$$
Thus, the lower bound in the theorem will follow if we show that there are at least
$b_px/\log\log x$ integers $m$ coprime to $2p$ with $m\le x/2^{g_p(m)+1}$.
Using Lemma~\ref{lem-2power}, we have $g_p(m)+1= v_2(\lambda(m))+v_2(p^2-1)$.
Further, it is easy to see that $2^{v_2(p^2-1)}\le2(|p|+1)$, with equality
when $|p|+1$ is a power of 2.

It follows from \cite[Section~2, Remark~1]{amicable} 
that uniformly for all $x\ge3$ and all positive integers $n$,
\begin{equation}
\label{eq-recipsum}
\sum_{\substack{r\le x\\r~{\rm prime}\\n\mid r-1}}\frac1r
=\frac{\log\log x}{\varphi(n)}+O\left(\frac{\log(2n)}{\varphi(n)}\right).
\end{equation}
We apply this with $n=2^{g_0+1}$, where
$g_0$ is the first integer with $2^{g_0}\ge4\log\log x$.  Thus, if
$\R$ is the set of primes $r\le x$ with $v_2(r-1)>g_0$,
we have for $x$ sufficiently large,
$$
\sum_{r\in\R} \frac1r<\frac13.
$$
Let $z=x/(25|p|\log\log x)$.  In $[1,z]$ there are $(\varphi(|p|)/(2|p|))z+O_p(1)$
integers coprime to $2p$.  And for a given value of $r\in\R$ there are 
at most $(\varphi(|p|)/(2|p|))z/r+O_p(1)$ numbers in $[1,z]$ coprime to $2p$
and divisible by $r$.
It follows that for $x$ sufficiently large depending on the choice of $p$,
there are at least
$$
\frac{\varphi(|p|)}{2|p|}z-\frac{\varphi(|p|)}{2|p|}z\sum_{r\in\R}\frac1r
+O_p\left(\sum_{r\in\R}1\right)
>\frac{\varphi(|p|)}{4|p|}z
$$
integers $m\le z$ coprime to $2p$ and not divisible by any prime $r\in\R$.
(We used that $|\R|=O(x/\log x)$ to estimate the $O$-term above.)

It remains to note that if $m\le z$, $m$ is coprime to $2p$, and
$m$ is not divisible by any prime in $\R$, then $v_2(\lambda(m))\le g_0$,
so that 
$$
2^{g_p(m)+1}\le 2^{v_2(\lambda(m))+v_2(p^2-1)}\le 2^{g_0}2^{v_2(p^2-1)}
\le 2^{g_0}\cdot2(|p|+1)\le 2^{g_0}\cdot3|p|<25|p|\log\log x.
$$
Thus, $2^{g_p(m)+1}m\in\B_{p,0}$ and $2^{g_p(m)+1}m\le x$, so that
$$
\B_{p,0}(x)\ge \frac{\varphi(|p|)}{100p^2}\frac{x}{\log\log x},
$$
for $x$ sufficiently large depending on the choice of $p$.
This completes our proof of the lower bound.

For the upper bound, it suffices to show that
$$
N(x):=\B_{p,0}(x)-\B_{p,0}(x/2)=O_p\left(\frac{x}{\log\log x}\right).
$$
(With this assumption, no two numbers $d$ counted can have the same
odd part.)
We shall assume that $p$ is not a square, the case when $p=p_0^{2^j}$
for some integer $p_0$ and $j\ge1$ being only slightly more complicated.
 From Proposition~\ref{prop-crit},
$N(x)$ is at most the number of odd numbers $m$ coprime to $p$
with $m\le x/2^{f_p(m)+1}$.  Let $N_k(x)$ be the number of odd numbers
$m\le x/2^{k+1}$ with $m$ coprime to $p$ and $f_p(m)=k$.  Then
$$
N(x)=\sum_kN_k(x)=\sum_{2^k\le\log\log x}N_k(x)+O\left(\frac{x}{\log\log x}\right).
$$
We now concentrate our attention on $N_k(x)$ with $2^k\le\log\log x$.  
If $f_p(m)=k$, then
$m$ is not divisible by any prime $r$ with $(p/r)=-1$ and $2^{k+1}\mid r-1$.
Then, using \eqref{eq-recipsum} and quadratic reciprocity,
$$
\sum_{\substack{r\le x\\(p/r)=-1\\2^{k+1}\mid r-1\\r~{\rm prime}}}
\frac1r=\frac{\log\log x}{2^{k+1}}+O_p\left(\frac{k}{2^k}\right).
$$
By Propostion~\ref{prop-sieve}, the number of integers $m\le x/2^{k+1}$ not divisible
by any such prime $r$ is at most
$$
O\left(\frac{x}{2^{k+1}}\exp\left(-\sum_r\frac1r\right)\right)
=O_p\left(\frac{x}{2^{k+1}}\exp\left(-\frac{\log\log x}{2^{k+1}}\right)\right).
$$
Summing this expression for $2^k\le \log\log x$ gives $O_p(x/\log\log x)$,
which completes the proof of Theorem~\ref{thm:dist0}.
\end{proof}

\begin{rem}
One might wonder if there is a positive constant $\beta_p$ such that
if $p$ is odd with $|p|>1$, then $\B_{p,0}(x)\sim\beta_px/\log\log x$
as $x\to\infty$.  Here we sketch an argument that no
such $\beta_p$ exists; that is, 
$$
0<\liminf_{x\to\infty}\frac{\B_{p,0}(x)}{x/\log\log x}
<\limsup_{x\to\infty}\frac{\B_{p,0}(x)}{x/\log\log x}<\infty.
$$
First note that but for $O_p(x/(\log x)^{1/2})$ values of $d\le x$
there is a prime $r\mid d$ with $(p/r)=-1$.
(We are assuming here that $p$ is not a square.)  For such values
of $d=2^jm$, with $m$ odd, we have $2\mid l_p(m)$, so that in
the notation above we have $f_p(m)=f'_p(m)\ge3$.  Thus it suffices
to count numbers $2^jm\le x$ with $m$ odd and $j>f_p(m)\ge3$.
Note that
$$
f_p(m)=v_2(l_p(m))+v_2(p^2-1)-1=v_2(l_p(m))+h_p-1,
$$
say.  Further,
$$
v_2(l_p(m))=\max_{r\mid m}v_2(l_p(r)),
$$
where $r$ runs over the prime divisors of $m$.  We have
\begin{multline*}
\{r\hbox{ prime}:v_2(l_p(r))=k\}=\\
\bigcup_{i\ge0}\{r\hbox{ prime}:
v_2(r-1)=k+i,~p\hbox{ is a }2^i\hbox{ power }\kern-8pt\pmod{r}
\hbox{ and not a }2^{i+1}\hbox{ power }\kern-8pt\pmod{r}\}.
\end{multline*}
For $k>(\log\log\log x)^2$, the density of primes $r\equiv1\pmod{2^k}$ is
so small that we may assume that no $d$ is divisible by such a prime $r$.
For $k$ below this bound, the density of primes $r$ with $v_2(l_p(r))=k$
is $1/(3\cdot2^{k-1})$.  Thus, there is a positive constant $c_{k,p}$
with $c_{k,p}\to1$ as $k\to\infty$ such that the density of integers
$m$ coprime to $2p$ and with $f_p(m)<k+h_p$ is asymptotically equal to
$$
c_p(\varphi(2|p|)/(2|p|))\exp(-(\log\log x)/(3\cdot2^k)),
$$
as $x\to\infty$.
Thus, the number of $m\le x/2^{k+h_p}$ coprime to $2p$ and with
$f_p(m)=k+h_p-1$ is asymptotically equal to
$$
c_{k,p}\frac{\varphi(2|p|)}{2|p|}\frac{x}{2^{k+h_p}}
\frac{\log\log x}{3\cdot2^k}\exp\left(-\frac{\log\log x}{3\cdot2^k}\right)
$$
as $x\to\infty$.  This expression then needs to be summed over $k$.
For $k$ small, the count is negligible because of the exp factor.
For $k$ larger, we can assume that the coefficients $c_{k,p}$ are all~1,
and then the sum takes on the form 
$$
\frac{\varphi(2|p|)}{2|p|2^{h_p}}\frac{x}{\log\log x}
\sum_k\frac{(\log\log x)^2}{3\cdot2^{2k}}\exp\left(-\frac{\log\log x}{3\cdot2^k}
\right).
$$
Letting this sum on $k$ be denoted $g(x)$, it remains to note that
$g(x)$ is bounded away from both 0 and $\infty$ yet does not tend to
a limit, cf.~\cite[Theorem~3.25]{Li}.
\end{rem}

To prove Theorem~\ref{thm:dist1}, we first establish the following
result.
\begin{prop}
\label{prop:crit}
Let $p$ be an integer with $|p|>1$.  Let $d$ be a positive integer
coprime to $p$ such that $d$ is divisible by odd primes $s$,
$t$ with 
$$
l_p(s)\equiv2\kern-5pt\pmod4,\quad l_p(t)\equiv1\kern-5pt\pmod2,
\quad
\langle p,-1\rangle_s\ne \U_s,\quad
\langle p,-1\rangle_t\ne \U_t.
$$
Assume that $4\mid l_p(d)$.  Then either $4\mid d$ and
$\frac12d+1\in\langle p\rangle_d$ or
$\langle p\rangle_d$ is not balanced.
\end{prop}
\begin{proof}
Let $k=l_p(d)$.
First assume that $4\mid d$ and $\frac12d+1\not\in\langle p\rangle_d$.
Let $2^\kappa$ be the
largest power of 2 in $k$.  Write
$d=2^jm$ where $m$ is odd, let $2^{\kappa_1}$ be the power of 2 in $l_p(m)$,
and let $2^{\kappa_2}=l_p(2^j)$.
Then $\kappa=\max\{\kappa_1,\kappa_2\}$.  Suppose that
$\kappa_2>\kappa_1$.  
We have $p^{k/2}\equiv 1\pmod m$ and 
$p^{k/2}\not\equiv1\pmod{2^j}$.  Since $4\mid k$, we have 
$p^{k/2}+1\equiv2\mod4$, and since $p^k-1=(p^{k/2}-1)(p^{k/2}+1)$,
we have $p^{k/2}\equiv1\pmod{2^{e-1}}$.  Thus, $p^{k/2}\equiv\frac12d+1\pmod d$,
contrary to our assumption.
Hence, we may assume that $\kappa=\kappa_1\ge\kappa_2$.  Note that
this inequality holds too in the case that $4\nmid d$, since then, $\kappa_2=0$.

We categorize the odd prime powers $r^a$ coprime to $p$ as follows.
\begin{itemize}
\item{} {\bf Type 1}: $\langle p,-1\rangle_{r^a}=\U_{r^a}$.
\item{} {\bf Type 2}: $\langle p,-1\rangle_{r^a}\ne \U_{r^a}$.
\item{} {\bf Type 3}: It is Type 2 and also $l_p(r^a)\equiv2\pmod4$.
\item{} {\bf Type 4}: It is Type 2 and also $l_p(r^a)$ is odd.
\end{itemize}
By assumption $d$ has at least one Type 3 prime power component
and at least one Type 4 prime power component.
We will show that $\langle p\rangle_d$
is not balanced in $\U_d$.
By Proposition~\ref{prop:non-vanishing}, it is sufficient to exhibit
an odd character $\chi\pmod d$ 
that is trivial at $p$  with
conductor $d'$ divisible by the same odd primes as are in $d$, and with
either $d\equiv2\pmod 4$ or $d/d'$ odd.

Let $r_1^{a_1}\| d$ where the power of 2 in $l_p(r_1^{a_1})$ is $2^{\kappa_1}$.
(Note that $r_1^{a_1}$ cannot be Type 3 nor Type 4, since
we have $\kappa_1=\kappa\ge2$, so that $4\mid l_p(r_1^{a_1})$.)
Consider the Type 1 prime powers in $d$, other than possibly $r_1^{a_1}$
in case it is of Type 1.  For each we take the quadratic character
and we multiply these together to get a character $\chi_1$ whose conductor
contains all of the primes involved in Type 1 prime powers, except possibly
$r_1$.

If $j\le1$, we let $\psi_{2^j}$ be the principal character mod $2^j$.
If $j\ge2$, let $\psi_{2^j}$ be a primitive character mod $2^j$ with
$\psi_{2^j}(p)=\zeta$, a primitive $2^{\kappa_2}$-th root of unity.  
Let $\chi_2=\chi_1\psi_{2^j}$.

We choose a character $\psi_{r_1^{a_1}}$ mod $r_1^{a_1}$ with
$\psi_{r_1^{a_1}}(p)=\chi_2(p)^{-1}$ if $\chi_2(p)\ne1$, and otherwise
we choose it so that $\psi_{r_1^{a_1}}(p)=-1$.  Thus, this character
is non-principal.  Let $\chi_3=\psi_{r_1^{a_1}}\chi_2$.  We now
have $\chi_3(p)=\pm1$.

If $\chi_3(p)=-1$ we use a Type 3 prime power $r_3^{a_3}\| d$ and
choose a 
character $\psi_{r_3^{a_3}}\pmod{r_3^{a_3}}$ with $\psi_{r_3^{a_3}}(p)=-1$.
Let $\chi_4=\chi_3\psi_{r_3^{a_3}}$.  If $\chi_3(p)=1$, we let
$\chi_4=\chi_3$.  We now have $\chi_4(p)=1$.

If $\chi_4(-1)=1$, we use a Type 4 prime power $r_4^{a_4}\|d$ and choose a
character $\psi_{r_4^{a_4}}\pmod{r_4^{a_4}}$ with $\psi_{r_4^{a_4}}(p)=1$
and $\psi_{r_4^{a_4}}(-1)=-1$.  Let $\chi_5=\chi_4\psi_{r^a}$.
If $\chi_4(-1)=-1$, we let $\chi_5=\chi_4$.

All remaining prime powers $r^a$ in $d$ are of Type 2.
For these we take non-principal characters that are trivial on
$\langle p,-1\rangle_{r^a}$, and multiply them in to $\chi_5$ to form $\chi_6$.
This is the character we are looking for, and so $\langle p\rangle_d$ is
not balanced.  This completes our proof.
\end{proof}

\begin{proof}[Proof of Theorem~\ref{thm:dist1}]
In the proof we shall assume that $p$ is neither a square nor
twice a square, showing in these cases that we may take
$\epsilon_p=1/16$.  The remaining cases are done with small adjustments
to the basic argument, but may require a smaller value for $\epsilon_p$.

Let $d\le x$ be coprime to $p$.
The set of primes $r\nmid p$ 
with $r\equiv1\pmod 4$ and for which 
$p$ is a quadratic nonresidue has density $1/4$, and in fact,
the sum of reciprocals of such primes $r\le x$ is $\frac14\log\log x+O_p(1)$. 
(This follows from either \eqref{eq-recipsum} and quadratic reciprocity
or from the Chebotarev density theorem.)
Thus by Proposition~\ref{prop-sieve}, 
the number of integers $d\le x$ not divisible by any of these primes
$r$ is $O_p(x/(\log x)^{1/4})$.  Thus, we may assume that $d$ is divisible
by such a prime $r$ and so that $4\mid l_p(d)$.  

Note that if $r\equiv5\pmod8$ and that $p$ is a quadratic
residue modulo $r$, but not a fourth power, then any $r^a$ is of
Type 3.  The density of these primes $r$ is $1/16$, by the Chebotarev
theorem, in fact, the sum of reciprocals of such primes $r\le x$ is
$\frac1{16}\log\log x+O_p(1)$.  
So the number of values of $d\in[3,x]$ not divisible by at least
one of them is $O_p(x/(\log x)^{1/16})$, using Proposition~\ref{prop-sieve}.  
Also note that if 
$r\equiv5\pmod8$ and $p$ is a nonzero fourth power modulo $r$, then
any $r^a$ is Type 4.  The density of these primes $r$ is also $1/16$,
and again the number of $d\in[3,x]$ not divisible by at least one of them
is $O_p(x/(\log x)^{1/16})$.

Thus, the number of values of $d\le x$ coprime to $p$ and not satisfying
the hypotheses of Proposition~\ref{prop:crit} is $O(x/(\log x)^{1/16})$.
This completes the proof of Theorem~\ref{thm:dist1}.
\end{proof}

\section{The average and normal order of the rank}

In this section we consider the average and normal order of the rank
of the curve $E_d$ given in Theorem~\ref{thm:CHU} as $d$ varies.

It is clear from Theorem~\ref{thm:CHU} that for $q$ odd,
$$\rk E_d(\Fq(u))\le
\begin{cases}
d-2&\text{if $d$ is even}\\
d-1&\text{if $d$ is odd}
\end{cases}$$
with equality when $d\in\B_{p,1}$ and
$q\equiv1\pmod d$.  

For all $q$ and $d>1$, it is known \cite[Prop.~6.9]{Brumer}
that
$$\rk E_d(\Fq(u))\le\frac{d}{2\log_q d}+O\left(\frac{d}{(\log_q d)^2}\right).$$
(Here $\log_q d$ is the logarithm of $d$ base $q$, i.e., $\log d/\log
q$.)  We do not include the details here, but this bound can be proved
directly for the the curves in Theorem~\ref{thm:CHU} using that theorem.
In addition, for $q$ odd, considering values of $d$ of the form 
$q^f+1$ for some positive integer $f$, and using
Theorem~\ref{thm:CHU}, we see that the main term in this inequality is
sharp for this family of curves.

We show below that although the average rank of $E_d(\Fq(u))$ is
large---its average for $d$ up to $x$ is at least $x^{1/2}$---for
``most'' values of $d$ the rank is much smaller.

\begin{thm}
\label{thm:average}
There is an absolute constant $\alpha>\frac12$ with the following
property.  For each odd prime $p$ and finite field $\Fq$ of
characteristic $p$, with $\Fq(u)$ and $E_d$ as in Theorem~\ref{thm:CHU},
we have
$$
x^\alpha\le\frac1x\sum_{d\le x}{\rm Rank}\,E_d(\Fq(u))
\le x^{1-\log\log\log x/(2\log\log x)}
$$
for all sufficiently large $x$ depending on the choice of $p$.
\end{thm}
\begin{proof}
This result follows almost immediately from \cite[Theorem~1]{PS}.
A result is proved there for the average value of the rank of curves
in a different family also parametrized by a positive integer $d$.
Using the notation from the present paper, if $d\in\B_{p,1}$ the
rank of the curve considered in \cite{PS} is within 4 of
\begin{equation}
\label{eq-ranksum}
\sum_{\substack{e\mid d\\e>2}}\frac{\varphi(e)}{l_q(e)}.
\end{equation}
We have $d\in\B_{p,1}$ implies that $e\in\B_{p,1}$ for all $e\mid d$ with
$e>2$.  By Theorem~\ref{thm:CHU}, 
formula \eqref{eq-ranksum} is exactly the rank of $E_d(\Fq(u))$
for $d\in\B_{p,1}$.  Since the proof of the lower bound $x^\alpha$
in~\cite{PS} uses only values of $d\in\B_{p,1}$, we have
the lower bound $x^\alpha$ in the present theorem.  

Since the rank of $E_d(\Fq(u))$ is bounded above by the formula
\eqref{eq-ranksum} whether or not $d$ is in $\B_{p,1}$, and in fact
whether or not $\langle p\rangle_d$ is balanced, the argument given in
\cite{PS} for the upper bound gives our upper bound here.
\end{proof}

\begin{thm}
\label{thm:normal}
For each odd prime $p$ and finite field $\Fq$ of
characteristic $p$, with $\Fq(u)$ and $E_d$ as in Theorem~\ref{thm:CHU},
we have but for $o(x/\log\log x)$ values of $d\le x$ with $d\in\B_p$ that
$$
{\rm Rank}\,E_d(\Fq(u))\ge(\log d)^{(1+o(1))\log\log\log d}
$$
as $x\to\infty$.  Further, assuming the GRH, we have but for $o(x/\log\log x)$
values of $d\le x$ with $d\in\B_p$ that
$$
{\rm Rank}\,E_d(\Fq(u))\le(\log d)^{(1+o(1))\log\log\log d}
$$
as $x\to\infty$.  Assuming the GRH, this upper bound holds but for $o(x)$
values of $d\le x$ coprime to $p$ as $x\to\infty$, regardless of whether
$d\in\B_p$.
\end{thm}
\begin{proof}
For $d\in\B_p$, Theorem~\ref{thm:CHU} implies that the rank of
$E_d(\Fq(u))$ is at least $\varphi(d)/l_q(d)\ge\varphi(d)/\lambda(d)$,
where $\lambda$ was defined in the previous section as the order
of the largest cyclic subgroup of $\U_d$.  It is shown in the proof
of Theorem 2 in \cite{EPS} that on a set of asymptotic density~1,
we have $\varphi(d)/\lambda(d)=(\log d)^{(1+o(1))\log\log\log d}$.
We would like to show this holds for almost all $d\in\B_p$.  Note
that we have $\varphi(m)/\lambda(m)=(\log m)^{(1+o(1))\log\log\log m}$
for almost all odd numbers $m$.  We have for all odd $m$ and
every integer $j\ge0$ that 
\begin{equation}
\label{eq-comp}
\frac{\varphi(m)}{\lambda(m)}\le
\frac{\varphi(2^jm)}{\lambda(2^jm)}
\le 2^j\frac{\varphi(m)}{\lambda(m)}.
\end{equation}
Thus, for almost all odd numbers $m$ we have for all nonnegative integers
$j$ with $2^j\le\log m$ that 
$\varphi(2^jm)/\lambda(2^jm))=(\log(2^jm))^{(1+o(1))\log\log\log(2^jm)}$.
Further, it follows from~\eqref{eq-recipsum} that but for a set of
odd numbers $m$ of asymptotic density~0, we have 
$v_2(\lambda(m))\le2\log\log\log m$.  It thus follows from 
Proposition~\ref{prop-crit} that for almost all odd numbers $m$ there
is some nonnegative $j$ with $2^jm\in\B_p$ and $2^j\le\log m$.
By Theorems~\ref{thm:dist0},~\ref{thm:dist1} almost all members of $\B_p$ are of
this form, and so we have the lower bound in the theorem.

For the upper bound we use an argument in \cite{LP}.  There,
Corollary 2 and the following remark imply that under the assumption
of the GRH, for almost all numbers $d$ coprime to $p$ we have 
$\varphi(d)/l_q(d)=(\log d)^{(1+o(1))\log\log\log d}$.  We use that
$\varphi(e)/l_q(e)\mid\varphi(d)/l_q(d)$ for $e\mid d$ and from the
normal order of the number-of-divisors function $\tau(d)$, that
most numbers $d$ have $\tau(d)\le\log d$.  It thus follows
from Theorem~\ref{thm:CHU} and the GRH that for almost all 
numbers $d$ coprime to $p$ that 
$$
{\rm Rank}E_d(\Fq(u))\le\tau(d)\frac{\varphi(d)}{l_q(d)}
\le(\log d)\frac{\varphi(d)}{l_q(d)}=(\log d)^{(1+o(1))\log\log\log d}.
$$

We would like to show as well that this inequality continues to hold
for almost all $d$ that are in $\B_p$.
As above, the GRH implies that for almost all
odd numbers $m$ coprime to $p$, 
we have $\varphi(m)/l_q(m)=(\log m)^{(1+o(1))\log\log\log m}$.
Since \eqref{eq-comp} continues to hold with $l_q$ in place of $\lambda$,
it follows that for almost all odd $m$ and for all $j$ with $1\le2^j\le\log m$,
that $\varphi(2^jm)/l_q(2^jm)=(\log(2^jm))^{(1+o(1))\log\log\log(2^jm)}$.
Again using the normal order of the number-of-divisors function $\tau$,
we have that for almost all odd $m$ and all $j$ with $1\le2^j\le\log m$
that $\tau(2^jm)\le\log m$.  Further, as we noted above,
from Theorems~\ref{thm:dist0},~\ref{thm:dist1},
it follows that almost all members $d$ of $\B_p$ are of the form $2^jm$ with
$m$ odd and $2^j\le\log m$.  The rank formula in Theorem~\ref{thm:CHU}
implies that the rank of $E_d(\Fq(u))$ is bounded above by
$\tau(d)\varphi(d)/l_q(d)$.  Thus, for almost all $d\in\B_p$ we have
the rank at most $(\log d)^{(1+o(1))\log\log\log d}$.  This
completes the proof.
\end{proof}


\begin{thebibliography}{1000}

\bibitem{Brumer}
A. Brumer, The average rank of elliptic curves. I, Invent. Math. {\bf 109}
(1992), 445--472.
 
\bibitem{CHU}
D.~Ulmer et al., 
Explicit points on the Legendre curve III,
In preparation.


\bibitem{EPS}
P. Erd\H os, C. Pomerance, and E. Schmutz,
Carmichael's lambda function,
Acta Arith.\ {\bf58} (1991), 363--385.

\bibitem{HR}
H. Halberstam and H.-E. Richert, Sieve methods, Academic Press, London, 1974.


\bibitem{Li}
S. Li,
On Artin's conjecture for composite moduli, Ph.D.\ thesis, U. Georgia, 1998.

\bibitem{LP}
S. Li and C. Pomerance,
On generalizing Artin's conjecture on primitive roots to composite moduli, 
J.\ Reine Angew.\ Math.\ {\bf556} (2003), 205--224. 

\bibitem{Marcus}
D.~Marcus,
Number fields,
Springer, New York, 1977.

\bibitem{MV}
H. L. Montgomery and R. C. Vaughan, 
Multiplicative number theory I. Classical theory,
Cambridge U. Press, Cambridge, 2007.

\bibitem{Moree}
P. Moree, On the divisors of $a^k+b^k$, Acta Arith.\ {\bf80} (1997), 197--212.

\bibitem{amicable}
C. Pomerance, On the distribution of amicable numbers,
J. Reine Angew.\ Math.\ {\bf293/294} (1977), 217--222.

\bibitem{PS}
C. Pomerance and I. E. Shparlinski,
Rank statistics for a family of elliptic curves over a function field,
Pure Appl.\ Math.\ Q., {\bf6} (2010), 21--40. 

\bibitem{U-L1}
D.~Ulmer, Explicit points on the Legendre curve,
Preprint (2009).

\end{thebibliography}
\end{document}